\documentclass[a4paper,10pt]{article}

\usepackage{mypackages}
\usepackage{mymacros}
\usepackage{mystyle}

\title{Coarse Homotopy on metric Spaces and their Corona}
\author{Elisa Hartmann}

\begin{document}

\maketitle

\begin{abstract}
This paper discusses properties of the Higson corona by means of a quotient on coarse ultrafilters on a proper metric space. We use this description to show that the corona functor is faithful. This study provides a Künneth formula for twisted coarse cohomology. We obtain the Gromov boundary of a hyperbolic proper geodesic metric space as a quotient of its Higson corona.
\end{abstract}

\tableofcontents

\section{Introduction}

The corona $\corona X$ of a metric space $X$ has been introduced in \cite{Protasov2003} and studied in \cite{Protasov2005},\cite{Protasov2011},\cite{Banakh2013},\cite{Protasov2015},\cite{Grzegrzolka2018a},\cite{Hartmann2019a}.

The Stone-\v Cech compactification is a functor $\beta$ from the category of completely regular spaces to the category of compact Hausdorff spaces. Note that by \cite[Theorem~2.1]{Baladze2018} if $X$ is a completely regular space and $G$ a group then
\[
\hat H^n_F(X;G)=\cohomology n {\beta X} G
\]
The left side denotes $n$-dimensional \v Cech type functional cohomology based on finite open covers and the right side denote $n$-dimensional \v Cech cohomology.

This resembles \cite[Corollary~35]{Hartmann2019a} where sheaf cohomology based on finite coarse covers of a metric space $X$ is related to sheaf cohomology on the corona $\corona X$. This property and other properties which we are going to discuss in this paper suggest that the corona functor is the Stone-\v Cech boundary version of a space in the coarse category.

We start with the first quite elementary property:

\begin{thma}
If $\mcoarse$ denotes the category of metric spaces and coarse maps modulo close and $\topology$ the category of topological spaces and continuous maps then the functor
\[
\nu':\mcoarse\to \topology
\]
is faithful.
\end{thma}

A direct consequence of this result is that $\nu'$ reflects isomorphisms.

We examine in which way the corona functor $\nu'$ is related to the Higson corona $\nu$ of \cite{Roe2003}. Originally the Higson corona has been defined on a proper metric space $X$ as the boundary of the compactification detemined by an algebra of bounded functions called the Higson functions. Already \cite{Protasov2005} showed that there exists a homeomorphism $\nu (X) =\corona X$. We provide an explicit homeomorphism and show $\nu,\nu'$ agree on morphisms too.

\begin{thma}
If $X$ is a proper metric space then there is a homeomorphism 
\[
\corona X\to \nu(X).
\]
Here the right side denote the Higson corona of \cite{Roe2003}. If $f:X\to Y$ is a coarse map between proper metric spaces then $\corona f,\nu(f)$ are homeomorphic (the same map pre-and postcomposed by a homeomorphism).
\end{thma}

The asymptotic product of two metric spaces has been introduced in \cite{Hartmann2019b} as the limit of a pullback diagram in the coarse category. Note \cite[Theorem 1]{Foertsch2003} shows the following: If $X,Y$ are hyperbolic coarsely proper coarsely geodesic metric spaces then $X\ast Y$ is hyperbolic coarsely proper coarsely geodesic and therefore its Gromov boundary $\partial (X\ast Y)$ is defined. There is a homeomorphism $\partial(X\ast Y)=\partial (X) \times \partial (Y)$ which is the main result of \cite{Foertsch2003}.

This paper shows if $X,Y$ are metric spaces then there is a homeomorphism $\corona X \times \corona Y=\corona{X\times Y}$. If $Y$ is coarsely geodesic coarsely proper then $\corona{X\ast Y}$ is the pullback of
\[
\xymatrix{
& \corona Y\ar[d]^{\corona {d(\cdot,q)}}\\
\corona X\ar[r]_{\corona{d(\cdot,p)}}
&\corona {\Z_+}
}
\]
Here $p\in X,q\in Y$ denote fixed points. Thus $\nu'$ preserves limits of this type. We obtain a coarse version of a Künneth formula for coarse cohomology with twisted coefficients:

\begin{thma}\name{Künneth formula}
Let $X,Y$ be metric spaces, $\sheaff$ a sheaf on $X$ and $\sheafg$ a sheaf on $Y$. Define a presheaf $\sheaff'$ on $X\times Y$ by
\[
U\mapsto \sheaff(p_1(U)).
\]
Then $\sheaff'$ is a sheaf on $X\times Y$. Similarly we can define a sheaf $\sheafg'$ on $X\times Y$. There is a homomorphism
\[
\bigoplus_{p+q=n}\cohomology p X \sheaff \otimes \cohomology q Y \sheafg \to \cohomology n {X\times Y} {\sheaff'\otimes \sheafg'}  
\]
Here $\sheaff'\otimes\sheafg'$ denotes the sheaf associated to the presheaf $U\mapsto \sheaff'(U)\otimes \sheafg'(U)$ for $U\s X\times Y$. If there is a $\sheaff$-acyclic coarse cover $\ucover$ of $X$ and a $\sheafg$-acyclic coarse cover $\vcover$ of $Y$ such that $\check C^q(\vcover,\sheafg)$ is torsion free for every $q$ and $\cohomology p \ucover\sheaff$ is torsion free for every $p$ then the homomorphism is an isomorphism.
\end{thma}

If $X$ is a hyperbolic proper geodesic metric space its Gromov boundary $\partial (X)$ is defined \cite{Benakli2002}. Since every Gromov function is a Higson function the Gromov boundary arises as a quotient of the Higson corona \cite{Roe2003}. We provide an explicit description of the quotient map and the induced topology on $\partial(X)$.

\begin{thma}
Let $X$ be a proper geodesic hyperbolic metric space. The relation $\sheaff\sim\sheafg$ if $\sheaff,\sheafg\in \closedop {\rho(\Z_+)}$ for some coarsely injective coarse map $\rho:\Z_+\to X$ is an equivalence relation on coarse ultrafilters and the mapping
\f{
q_X:\corona X&\to\partial (X)\\
\sheaff&\mapsto [\rho]\;\;\; \sheaff\in\closedop{\rho(\Z_+)}
}
to the Gromov boundary $\partial (X)$ of $X$ is continuous and respects $\sim$. The induced map on the quotient associated to $\sim$ is a homeomorphism.

If $A\s X$ is a subset then
\[
\partial_X A:=\{[\rho]:\rho(\Z_+)\close A\}
\]
is closed in $\partial (X)$. The $((\partial_X A)^c)_{A\s X}$ constitute a basis for the topology on $\partial (X)$. 
\end{thma}

This result implies there is a larger class of morphisms in the coarse category for which the Gromov boundary is a functor. Originally coarse equivalences were shown to induce continuous maps between Gromov boundaries. If $f:X\to Y$ is a coarse map between hyperbolic proper geodesic metric spaces with the property that for every coarsely injective coarse map $\rho:\Z_+\to X$ the map $f\circ \rho$ is coarsely injective coarse then $f$ induces a map between Gromov boundaries.

\section{Metric Spaces}

\label{sec:metric}
\begin{defn}
 Let $(X,d)$ be a metric space. Then the \emph{coarse structure associated to $d$} on $X$ consists of those subsets $E\s X^2$ for which
 \[
  \sup_{(x,y)\in E}d(x,y)<\infty.
 \]
 We call an element of the coarse structure \emph{entourage}.  In what follows we assume the metric $d$ to be finite for every $(x,y)\in X^2$.
\end{defn}

\begin{defn}
 A map $f:X\to Y$ between metric spaces is called \emph{coarse} if
 \begin{itemize}
  \item $E\s X^2$ being an entourage implies that $\zzp f E$ is an entourage \emph{(coarsely uniform)};
  \item and if $A\s Y$ is bounded then $\iip f A$ is bounded \emph{(coarsely proper)}.
 \end{itemize}
 Two maps $f,g:X\to Y$ between metric spaces are called \emph{close} if
 \[
 f\times g(\Delta_X)
 \]
 is an entourage in $Y$. Here $\Delta_X$ denotes the diagonal in $X^2$.
\end{defn}

\begin{notat}
 A map $f:X\to Y$ between metric spaces is called
 \begin{itemize}
 \item \emph{coarsely surjective} if there is an entourage $E\s Y^2$ such that 
 \[
  E[\im f]=Y;
 \]
 \item \emph{coarsely injective} if for every entourage $F\s Y^2$ the set $\izp f F$ is an entourage in $X$.
\end{itemize}
  Two subsets $A,B\s X$ are called \emph{not coarsely disjoint} if there is an entourage $E\s X^2$ such that the set
  \[
  E[A]\cap E[B]
  \]
  is not bounded. We write $A\close B$ in this case.
  
  Two subsets $A,B\s X$ are called \emph{asymptotically alike} if there is an entourage $E\s X^2$ such that
  \[
  E[A]=B.
  \]
  We write $A\lambda B$ in this case.
\end{notat}

\begin{rem}
 We study metric spaces up to coarse equivalence. A coarse map $f:X\to Y$ between metric spaces is a \emph{coarse equivalence} if
 \begin{itemize}
  \item There is a coarse map $g:Y\to X$ such that $f\circ g$ is close to $id_Y$ and $g\circ f$ is close to $id_X$.
 \item or equivalently if $f$ is both coarsely injective and coarsely surjective.
 \end{itemize}
\end{rem}

\begin{defn}
A metric space is called \emph{coarsely proper} if it is coarsely eqivalent to a proper metric space. It is called \emph{coarsely geodesic} if it is coarsely equivalent to a geodesic metric space.
\end{defn}

\begin{notat}
If $X$ is a metric space and $U_1,\ldots, U_n\s X$ are subsets then $(U_i)_i$ are said to \emph{coarsely cover} $X$ if for every entourage $E\s X^2$ the set 
\[
E[U_1^c]\cap \cdots \cap E[U_n^c]
\]
is bounded.
\end{notat}

\section{The Corona Functor}

\begin{defn}
\label{defn:cu}
If $X$ is a metric space a system $\sheaff$ of subsets of $X$ is called a \emph{coarse ultrafilter} if
\begin{enumerate}
\item $A,B\in \sheaff$ then $A\close B$.
\item $A,B\s X$ are subsets with $A\cup B\in\sheaff$ then $A\in\sheaff$ or $B\in\sheaff$.
\item $X\in\sheaff$.
\end{enumerate}
\end{defn}

\begin{lem}
\label{lem:cucoarsemap}
If $f:X\to Y$ is a coarse map between metric spaces and $\sheaff$ is a coarse ultrafilter on $X$ then
\[
f_*\sheaff:=\{A\s Y:\iip f A\in\sheaff\}
\]
is a coarse ultrafilter on $Y$.
\end{lem}
\begin{proof}
see \cite{Hartmann2019a}.
\end{proof}

\begin{defn}
\label{defn:asymptoticallyalikecu}
We define a relation on coarse ultrafilters on $X$: two coarse ultrafilters $\sheaff,\sheafg$ are \emph{asymptotically alike}, written $A\lambda B$ if for every $A\in \sheaff,B\in \sheafg$:
\[
A\close B.
\]
\end{defn}

\begin{rem}
By \cite{Hartmann2019a} the relation $\lambda$ is an equivalence relation on coarse ultrafilters on $X$. If two coarse ultrafilters $\sheaff,\sheafg$ on $X$ are asymptotically alike and $f:X\to Y$ is a coarse map to a metric space $Y$ then $f_*\sheaff\lambda f_*\sheafg$ on $Y$.
\end{rem}

\begin{defn}
\label{defn:topologyoncu}
Let $X$ be a metric space. Denote by $\nu'(X)$ the set of coarse ultrafilters modulo asymptotically alike on $X$. The relation $\close$ on subsets of $\corona X$ is defined as follows: Define for a subset $A\s X$:
\[
\closedop A=\{[\sheaff]\in\corona X:A\in \sheaff\}
\]
Then $\pi_1\notclose \pi_2$ if and only if there exist subsets $A,B\s X$ such that $A\notclose B$ and $\pi_1\s \closedop A,\pi_2\s\closedop B$. 
\end{defn}

\begin{rem}
The relation $\close$ on subsets of $\corona X$ defines a proximity relation on $\corona X$ which induces a compact topology. By \cite{Hartmann2019a} the mapping $f_*$ between coarse ultrafilters induces a continuous map $\corona f$ between the quotients. Thus $\nu'$ is a functor mapping coarse metric spaces to compact topological spaces.

The topology on $\corona X$ is generated by $(\closedop A)^c_{A\s X}$. Coarse covers determine the finite open covers \cite{Hartmann2019a}.
\end{rem}

\section{On Morphisms}

\begin{lem}
\label{lem:coarseproximitymaps}
Let $f:X\to Y$ be a map between metric spaces. Then
\begin{enumerate}
\item $f$ is a coarse map if 
\begin{itemize}
\item $B\s X$ is bounded then $f(B)$ is bounded.
\item if for every subsets $A,B\s X$ the relation $A\close B$ implies $f(A)\close f(B)$.
\end{itemize}
\item if $f$ is coarse then $f$ is coarsely injective if $A\notclose B$ implies $f(A)\notclose f(B)$.
\item $f$ is coarsely surjective if the relation $f(X)\notclose B$ in $Y$ implies $B$ is bounded.
\end{enumerate}
\end{lem}
\begin{proof}
\begin{enumerate}
\item First we show $f$ is coarsely proper. If $B\s Y$ is bounded then $B\notclose Y$. This implies $\iip f B\notclose X$. Thus $\iip f B$ is bounded.

Now we show $f$ is coarsely uniform: Suppose $A,B\s X$ are two subsets with $f(A)\bar \lambda f(B)$. Then there is an unbounded subset $S\s f(A)$ with $S\notclose f(B)$ or there is an unbounded subset $T\s f(B)$ with $T\notclose f(A)$. Assume the former. Then $\iip f S\notclose B$. Since $f$ maps bounded sets to bounded sets the set $\iip f S\cap A$ is unbounded. Thus $A\bar\lambda B$. Thus we have shown $A\lambda B$ implies $f(A)\lambda f(B)$. By \cite[Theorem~2.3]{Kalantari2016} we can conclude that $f$ is coarsely uniform.
\item This is \cite[Lemma~41]{Hartmann2019a}.
\item easy.
\end{enumerate}
\end{proof}

\begin{thm}
If $f,g:X\to Y$ are two coarse maps between metric spaces and $\corona f=\corona g$ then $f,g$ are close.
\end{thm}
\begin{proof}
Suppose $f,g$ are not close. By \cite[Proposition~2.15]{Kalantari2016} there is a subset $A\s X$ with $f(A)\bar\lambda g(A)$. This implies there is a subset $S\s A$ with $f(S)\notclose g(S)$. Now by \cite[Proposition~4.7]{Grzegrzolka2018a} there is a coarse ultrafilter $\sheaff$ on $X$ with $S\in\sheaff$. Then $f(S)\in f_*\sheaff$ and $g(S)\in g_*\sheaff$. Since $f(S)\notclose g(S)$ this implies $f_*\sheaff\not= g_*\sheaff$. Thus $\corona f,\corona g$ are not the same map.
\end{proof}

\begin{cor}
\label{cor:faithful}
If $\mcoarse$ denotes the category of metric spaces and coarse maps modulo close and $\topology$ the category of topological spaces and continuous maps then the functor
\[
\nu':\mcoarse\to \topology
\]
is faithful.
\end{cor}

\begin{cor}
\label{cor:nureflectsepisandmonos}
The functor $\nu':\mcoarse\to \topology$ reflects epimorphisms and monomorphisms.
\end{cor}
\begin{proof}
It is general theory that a faithful functor reflects epimorphisms and monomorphisms. This fact can also be found in \cite[Exercise~1.6.vii]{Riehl2017}. Since by Corollary~\ref{cor:faithful} the functor $\nu'$ is faithful the result follows.
\end{proof}

\begin{cor}
\label{cor:nureflectsisos}
The functor $\nu':\mcoarse\to \topology$ reflects isomorphisms.
\end{cor}
\begin{proof}
Suppose $f:X\to Y$ is a coarse map between metric spaces such that $\corona f$ is an isomorphism in $\topology$. Then $\corona f$ is both a monomorphism and an epimorphism. The proof of \cite[Theorem~40]{Hartmann2019a} can be generalized to hold for metric spaces. Then the map $f$ is coarsely surjective. By Corollary~\ref{cor:nureflectsepisandmonos} the map $f$ is a monomorphism in $\mcoarse$. By a proof similar to the one of \cite[Proposition~3.A.16]{Cornulier2016} every monomorphism is coarsely injective. Since $f$ is coarsely injective and coarsely surjective it is a coarse equivalence.
\end{proof}

\begin{thm}
If $X$ is a proper metric space then there is a homeomorphism 
\[
\corona X\to \nu(X).
\]
Here the right side denote the Higson corona of \cite{Roe2003}. If $f:X\to Y$ is a coarse map between proper metric spaces then $\corona f,\nu(f)$ are homeomorphic (the same map pre-and postcomposed by a homeomorphism).
\end{thm}
\begin{proof}
Let $X$ be a proper metric space. First we show that $h'(X):=X\sqcup \corona X$ is a compactification of $X$: Closed sets on $h'(X)$ are generated by $(\bar A\cup \closedop A)_{A\s X}$. We show this topology is compact: If $(\bar A_i\cup \closedop{A_i})^c_i$ is an open cover of $h'(X)$ then there is a subcover 
\[
(\bar A_1 \cup \closedop{A_1})^c_1,\ldots,(\bar A_n\cup \closedop{A_n})^c
\]
such that $\closedop{A_1}^c,\ldots,\closedop{A_n}^c$ is a cover of $\corona X$. Now this implies $A_1^c,\ldots,A_n^c$ are a coarse cover of $X$. Thus $\bar A_1\cap\cdots\cap \bar A_n$ is both bounded and closed. Then there is a subcover 
\[
(\bar A_{n+1}\cup\closedop{A_{n+1}})^c,\ldots,(\bar A_{n+m}\cup\closedop{A_{n+m}})^c
\]
of $(\bar A_i\cup \closedop {A_i})^c_i$ such that $\bar A_{n+1}^c,\ldots,\bar A_{n+m}^c$ covers $\bar A_1\cap\cdots\cap \bar A_n$. Then
\[
(\bar A_1\cup\closedop {A_1})^c,\ldots,(\bar A_{n+m}\cup \closedop{A_{n+m}})^c
\]
are a subcover of $(\bar A_i\cup\closedop{A_i})^c_i$ that cover $h'(X)$.

Now $X,\corona X$ both appear as subspaces of $h'(X)$. We show the inclusion $X\to h'(X)$ is dense:
\f{
\bar X^{h'}
&=\bigcap_{\bar A\cup \closedop A\z X}(\bar A\cup \closedop A)\\
&=X\cup \closedop X\\
&=h'(X).
}

The Higson compactification $h(X)$ is determined by the $C^*$-algebra of Higson functions whose definition we now recall from \cite{Roe2003}: A bounded continuous function $\varphi:X\to \R$ is called \emph{Higson} if the function
\f{
d\varphi:X^2 &\to\R\\
 (x,y)&\mapsto \varphi(y)-\varphi(x)
}
when restricted to $E$ vanishes to infinity for every entourage $E\s X^2$.

Note \cite[Proposition~1]{Protasov2005} shows Higson functions on $X$ can be extended to $h'(X)$. For the convenience of the reader we recall it.

Without loss of generality assume that $X$ is $R$-discrete for some $R> 0$. Then every coarse ultrafilter $\sheaff$ on $X$ is determined by an ultrafilter $\sigma$ on $X$ by the proof of \cite[Theorem~17]{Hartmann2019a}. If $\sigma$ is an ultrafilter on $X$ then a bounded continuous function $\varphi:X\to \R$ determines an ultrafilter $\varphi_*\sigma:=\{A:\iip \varphi A\in\sigma\}$ on $\R$. Since the image of $\varphi$ is bounded and therefore relatively compact the ultrafilter $\varphi_*\sigma$ converges to a point $\sigma-\lim \varphi\in\R$. 

If two ultrafilters $\sigma,\tau$ induce asymptotically alike coarse ultrafilters and $\varphi$ is a Higson function then $\sigma-\lim \varphi =\tau-\lim \varphi$: Suppose $\sigma-\lim \varphi\not=\tau-\lim \varphi$. Then there exist neighborhoods $U\ni \sigma-\lim \varphi$ and $V\ni \tau-\lim \varphi$ such that $d(U,V)>0$. Let $E\s X^2$ be an entourage. Then
\f{
d\varphi: \iip \varphi U\times \iip \varphi V\cap E&\to \R\\
(x,y)&\to \varphi(y)-\varphi(x)
}
vanishes at infinity. Since $d(U,V)>0$ this implies that $\iip \varphi U\times \iip \varphi V\cap E$ is bounded. Now $E$ was an arbitrary entourage thus $\iip \varphi U,\iip \varphi V$ are coarsely disjoint. Since $\iip \varphi U\in \sigma,\iip \varphi V\in \tau$ the ultrafilters $\sigma,\tau$ induce coarse ultrafilters which are not asymptotically alike. 

If $\sheaff$ is a coarse ultrafilter on $X$ induced by an ultrafilter $\sigma$ and $\varphi$ a Higson function then denote by $\sheaff-\lim \varphi$ the point $\sigma-\lim \varphi$ in $\R$. By the above $\sheaff-\lim \varphi$ is well defined modulo asymptotically alike of $\sheaff$. 

If $\varphi:X\to \R$ is a Higson function then there is an extension
\begin{align*}
\hat \varphi:h'(X) &\to \R\\
x&\mapsto \begin{cases}
\varphi(x) & x\in X\\
\sheaff-\lim \varphi & x=\sheaff\in\corona X
\end{cases}
\end{align*}
we have shown $\hat \varphi$ is well defined. Now we show $\hat \varphi$ is continuous: Let $A\s \R$ be a closed set. If $\sheaff-\lim \varphi\in A$ fix an ultrafilter $\sigma$ on $X$ that induces $\sheaff$. Then $\iip \varphi A\in \sigma$. This implies $\sheaff\in\closedop{\iip \varphi A}$. On the other hand if $\sheaff\in\closedop{\iip \varphi A}$ then there is an ultrafilter $\sigma$ on $X$ with $\iip \varphi A\in\sigma$ that induces $\sheaff$. This implies $\sigma-\lim \varphi\in A$, thus $\sheaff-\lim \varphi\in A$. Now
\f{
\iip {\hat \varphi} A
&=\iip \varphi A \cup \{\sheaff:\sheaff-\lim \varphi\in A\}\\
&=\iip \varphi A\cup \closedop{\iip \varphi A}
}
is closed.

Denote by $(C_h(X))^{h'}$ the set of extensions of Higson functions on $X$ to $h'(X)$. By \cite{Ball1983} the $C^*$-algebra of Higson functions $C_h(X)$ determines the compactification $h'(X)$ if and only if $(C_h(X))^{h'}$ separates points of $\corona X$.

We show $(C_h(X))^{h'}$ separates points of $\corona X$: Let $\sheaff,\sheafg\in \corona X$ be two coarse ultrafilters with $\sheaff\bar\lambda \sheafg$. Then there exist elements $U\in\sheaff,V\in\sheafg$ with $U\notclose V$. Without loss of generality assume that $U,V$ are disjoint such that $d(x,U)+d(x,V)\not=0$ for every $x\in X$. Then define a function
\f{
\varphi:X&\to \R\\
x&\mapsto \frac{d(x,U)}{d(x,U)+d(x,V)}
}
By \cite[Lemma~2.2]{Dranishnikov1998} the function $d\varphi|_E$ vanishes to infinity for every entourage $E\s X^2$. Now $\varphi|_U\equiv 0$ and $\varphi|_V\equiv 1$. This implies $\sheaff-\lim \varphi=0$ and $\sheafg-\lim \varphi=1$.

If $f:X\to Y$ is a coarse map between $R$-discrete for some $R>0$ proper metric spaces and $\varphi:Y\to \R$ a Higson function then $\varphi\circ f:X\to \R$ is a Higson function: Since $X$ is $R$-discrete the map $f$ is continuous, therefore $\varphi\circ f$ is continuous. The map $\varphi\circ f$ is bounded since $\varphi$ is bounded. Let $E\s X^2$ be an entourage and $\varepsilon>0$ a number. Then $\zzp f E\s Y^2$ is an entourage. This implies $(d\varphi)|_{\zzp f E}$ vanishes at infinity. Thus there is a compact set $K\s Y$ such that
\[
|d(\varphi(x,y)|<\varepsilon
\] 
whenever $(x,y)\in \zzp f E\cap (K^2)^c$. Since $K$ is bounded the set $\iip f K\s X$ is bounded. The set $\iip f K$ is finite since $X$ is $R$-discrete and therefore $\iip f K$ is compact. Then
\[
|d(\varphi\circ f)(x,y)|<\varepsilon
\]
whenever $(x,y)\in E\cap (\iip f K)^2$.

Now we provide an explicit homeomorphism $\nu(X)\to \corona X$: Denote by 
\f{
e_{C_h(X)}:Z&\to \R^{C_h(X)}\\
x&\mapsto (\varphi(x))_\varphi
}
the evaluation map for $X$.

 Note $e_{C_h(X)}$ is a topological embedding and $\nu(X):=\overline{e_{C_h(X)}(X)}\ohne e_{C_h(X)}(X)$ by \cite{Ball1983}. A point $p\in\nu(X)$ is represented by a net $(x_i)_i$ such that for every Higson function $\varphi\in C_h(X)$ the net $\varphi(x_i)_i$ converges in $\R$. Define $F_i:=\{x_j:j\ge i\}$ for every $i$. Then $\sigma:=\{F_i:i\}$ is a filter on $X$ such that $\varphi_*\sigma$ converges to $\lim_i\varphi(x_i)$ for every Higson function $\varphi$ on $X$. An ultrafilter $\sigma'$ which is finer that $\sigma$ determines a coarse ultrafilter $\sheaff$. We have shown above that the association $\Phi_X:p\mapsto \sheaff$ is well defined modulo asymptotically alike. 

Now we show the map $\Phi_X$ is injective: Let $p,q\in \nu(X)$ be two points. If $\Phi_X(p)=\Phi_X(q)$ then $\Phi_X(p)-\lim \varphi=\Phi_X(q)-\lim \varphi$ for every Higson function $\varphi$. This implies $p=q$ in $\R^{C_h(X)}$. 

We show $\Phi_X$ is surjective: If $\sigma$ is an ultrafilter on $X$ that determines a coarse ultrafilter on $X$ then there is a net $(x_i)_i$ on $X$ which constitutes a section of $\sigma$. Since $\varphi(x_i)_i$ is a section of $\varphi_*\sigma$ for every Higson function $\varphi$ the net $\varphi(x_i)_i$ converges to $\sigma-\lim \varphi$ in $\R$. Thus $(x_i)_i$ converges to a point in $\nu(X)$.

Now we show $\Phi_X$ is continuous: If $A\s X$ is a subset then $\iip {\Phi_X} {\closedop A}$ is a subset of $\nu(X)$. We show it is closed. If $p\in \iip {\Phi_X} {\closedop A}$ then there is a net $(x_i)_i\s X$ that converges to $p$. The net $(x_i)_i$ is a section of an ultrafilter $\sigma$ with $A\in\sigma$. Thus there exists $i$ with $x_j\in A$ for every $j\ge i$. If on the other hand $(x_i)_i$ is a net in $X$ and there exists $i$ with $x_j\in A$ for every $j\ge i$ then $(x_i)_i$ is a section of an ultrafilter $\sigma$ on $X$ with $A\in \sigma$. This implies if $(x_i)_i$ converges to $p\in\nu(X)$ then $p\in\iip {\Phi_X}{\closedop A}$. Thus we have shown
\[
\iip {\Phi_X} {\closedop A}=\overline{e_{C_h(X)}(A)}\ohne e_{C_h(X)}(A)
\]
is closed. This way we have obtained that $\Phi_X$ is a homeomorphism.

Now we define a map 
\f{
f_*:\R^{C_h(X)}&\to \R^{C_h(Y)}\\
(x_\varphi)_{\varphi\in C_h(X)}&\mapsto (x_{\varphi\circ f})_{\varphi\in C_h(Y)}
}
We show $f_*(\overline{e_{C_h(X)}(X)})\s \overline{e_{C_h(Y)}(Y)}$: If $(x_\varphi)_\varphi\in\overline{e_{C_h(X)}(X)}$ then there is a net $(x_i)_i\s X$ such that $\lim_i\varphi(x_i)= x_\varphi$ for every $\varphi\in C_h(X)$. Then $f(x_i)_i\s Y$ is a net such that $\lim_i\varphi(f(x_i)) = x_{\varphi\circ f}$ for every $\varphi\in C_h(Y)$.

Now $\nu(f):=f_*|_{\overline{e_{C_h(X)}(X)}\ohne e_{C_h(X)}(X)}$. Then
\[
\nu(f)=\ii\Phi_Y\circ\corona f\circ \Phi_X.
\]
\end{proof}

\section{A K\"unneth Formula}

This is \cite[Definition~25]{Hartmann2019b}:
\begin{defn}\name{asymptotic product}
\label{defn:asymptoticproduct}
 If $X$ is a metric space and $Y$ a coarsely geodesic coarsely proper metric space fix points $p\in X$ and $q\in Y$ and a constant $R\ge 0$ large enough. Then the \emph{asymptotic product}
$X\ast Y$ of $X$ and $Y$ is defined by
\[
  X\ast Y:=\{(x,y)\in X\times Y: |d(p,x)-d(q,y)|\le R\}
\]
 as a subspace of $X\times Y$. We define the projection $p_1:X\ast Y\to X$ by $(x,y)\mapsto x$ and the projection $p_2:X\ast Y\to Y$ by $(x,y)\mapsto y$. Note that the projections are coarse maps. In what follows we denote by $d(p,\cdot),d(q,\cdot)$ coarse maps $X\to \R_+,Y\to \R_+$ defined by $x\in X\mapsto d(p,x),y\in Y\mapsto d(q,y)$.  
\end{defn}

\begin{rem}
Let $X,Y$ be metric spaces of which $Y$ is coarsely geodesic coarsely proper. Now $X\ast Y$ of Definition~\ref{defn:asymptoticproduct} is determined by points $p\in X,q\in Y$ and constant $R\ge 0$. By \cite[Lemma~26]{Hartmann2019b} the space $X\ast Y$ does not depend on the choice of $p,q,R$ up to coarse quivalence. By \cite[Lemma~27]{Hartmann2019b} the diagram
\[
 \xymatrix{
 X\ast Y \ar[d]_{p_1}\ar[r]^{p_2}
 & Y\ar[d]^{d(q,\cdot)}\\
 X\ar[r]_{d(p,\cdot)}
 &\R_+
 }
 \]
 is a pullback diagram in $\coarse$.
\end{rem}

\begin{lem}
\label{lem:asymptoticproductprops}
Let $X,Y$ be metric spaces with $Y$ coarsely geodesic coarsely proper. The following statements hold:
\begin{enumerate}
\item If $A\s X, B\s Y$ are subsets then $(A\times B)\cap (X\ast Y)$ is bounded if $A$ is bounded or $B$ is bounded.
\item If $(U_i)_i$ is a coarse cover of $X$ and $(V_j)_j$ a coarse cover of $Y$ then $((U_i\times V_j)\cap (X\ast Y))_{ij}$ is a coarse cover of $X\ast Y$
\item Let $\sheaff,\sheafg$ be coarse ultrafilters on $X,Y$ respectively with $d(p,\cdot)_*\sheaff\lambda d(q,\cdot)_*\sheafg$. Choose the constant of $X\ast Y$ large enough. Then
\[
\sheaff\ast\sheafg:=\{(A\times B)\cap (X\ast Y):A\in\sheaff,B\in\sheafg\}
\]
is a coarse ultrafilter on $X\ast Y$.
\end{enumerate}
\end{lem}
\begin{proof}
\begin{enumerate}
\item Suppose $A$ is bounded. Then $(x,y)\in A\ast Y$ implies $x\in A$ and $|d(x,p)-d(y,q)|\le R$. Let $S\ge 0$ be such that $A\s B(p,S)$. Then $y\in B(q,R+S)$. Thus $A\ast Y$ is bounded. Similarly if $B$ is bounded then $X\ast B$ is bounded.
\item Let $E\s (X\ast Y)^2$ be an entourage. Then
\f{
\bigcap_{ij} E[(U_i\times V_j)^c\cap (X\ast Y)]
&\s\bigcap_{ij} E[(U_i\times V_j)^c]\cap (X\ast Y)\\
&=\bigcap_{ij} (E[U_i^c\times Y]\cup E[X\times V_j^c])\cap (X\ast Y)\\
&= (\bigcap_i E[U_i^c\times Y]\cap (X\ast Y))\cup (\bigcap_j E[X\times V_j^c]\cap (X\ast Y))
}
is bounded. Thus $((U_i\times V_j)\cap (X\ast Y))_{ij}$ is a coarse cover of $X\ast Y$.

Alternative proof: $(\iip{p_1}{U_i}\cap \iip{p_2}{V_j})_{ij}$.
\item Let $i:X\ast Y\to X\times Y$ be the inclusion. At first we prove
\[
i_*(\sheaff\ast\sheafg)=\{A\times B:A\in\sheaff,B\in\sheafg\}
\]
is a coarse ultrafilter on $X\times Y$. We check the axioms of a coarse ultrafilter on $i_*(\sheaff\ast\sheafg)$:
\begin{enumerate}
\item If $A_1\times B_1,A_2\times B_2\in i_*(\sheaff\ast\sheafg)$ then $A_1,A_2\in\sheaff,B_1,B_2\in\sheafg$. This implies $A_1\close A_2$ in $X$ and $B_1\close B_2$ in $Y$. Then $A_1\times B_1\close A_2\times B_2$ in $X\times Y$. 
\item Let $A_1\times B_1,A_2\times B_2\s X\times Y$ be two subsets with $(A_1\times B_1)\cup (A_2\times B_2)\in i_*(\sheaff\ast \sheafg)$. Since $(A_1\cup A_2)\times (B_1\cup B_2)\z (A_1\times B_1)\cup (A_2\times B_2)$ this implies $(A_1\cup A_2)\times (B_1\cup B_2)\in i_*(\sheaff\ast\sheafg)$. Thus $(A_1\cup A_2)\in\sheaff,(B_1\cup B_2)\in\sheafg$. This implies $A_1\in\sheaff$ or $A_2\in\sheaff$. Then $A_1\times (B_1\cup B_2)\in i_*(\sheaff\ast\sheafg)$ or $A_2\times (B_1\cup B_2)\in i_*(\sheaff\ast \sheafg)$. Suppose $A_1\times (B_1\cup B_2)\in i_*(\sheaff\ast \sheafg)$. Since $A_1\times B_1$ is maximal among factors of two subsets of $X,Y$ contained in $A_1\times (B_1\cup B_2),(A_1\times B_1)\cup (A_2\times B_2)\in i_*(\sheaff\ast\sheafg)$ we obtain $A_1\times B_1\in i_*(\sheaff\ast \sheafg)$.
\item $X\times Y\in i_*(\sheaff\ast \sheafg)$ since $X\in\sheaff,Y\in\sheafg$.
\end{enumerate}
Let $A\times B\in i_*(\sheaff\ast\sheafg)$ be an element. Since $d(p,\cdot)_*\sheaff\lambda d(q,\cdot)_*\sheafg$ the sets $d(p,\cdot)(A),d(q,\cdot)(B)$ are close in $\R_+$. Thus there exists an $R\ge 0$ and unbounded subsets $A'\s A,B'\s B$ with
\[
|d(p,a)-d(q,b)|\le R
\]
for $a\in A',b\in B'$. Thus we have shown $A\times B\close X\ast Y$. Choose the constant of $X\ast Y$ large enough then $X\ast Y\in i_*(\sheaff\ast\sheafg)$. We can thus restrict $i_*(\sheaff\ast \sheafg)$ to $X\ast Y$ and obtain $\sheaff\ast \sheafg=(i_*(\sheaff\ast\sheafg))|_{X\ast Y}$. This way we have shown $\sheaff\ast \sheafg$ is a coarse ultrafilter.
\end{enumerate}
\end{proof}

\begin{thm}
Let $X,Y$ be metric spaces with $Y$ coarsely geodesic coarsely proper. Define 
 \[
\corona X\ast \corona Y:=\{(\sheaff,\sheafg)\in\corona X\times \corona Y:\corona {d(p,\cdot)}(\sheaff)=\corona {d(q,\cdot)} (\sheafg)\}
\]
Then the map
\[
\langle\corona{p_1},\corona{p_2}\rangle:\corona{X\ast Y}\to\corona X\ast \corona Y
\]
is a homeomorphism.
\end{thm}
\begin{proof}
We prove $\langle\corona{p_1},\corona{p_2}\rangle$ is well defined: Let $\sheaff$ be a coarse ultrafilter on $X\ast Y$ then $p_{1*}\sheaff,p_{2*}\sheaff$ are coarse ultrafilters on $X,Y$, respectively. Since $d(p,\cdot)\circ p_1,d(q,\cdot)\circ p_2$ are close the coarse ultrafilters $d(p,\cdot)_*p_{1*}\sheaff,d(q,\cdot)_*p_{2*}\sheaff$ are asymptotically alike. Thus we have shown $(p_{1*}\sheaff,p_{2*}\sheaff)\in \corona X\ast\corona Y$.

Now we prove $\langle \corona{p_1},\corona{p_2}\rangle$ is surjective: Let $(\sheaff,\sheafg)\in\corona X\ast\corona Y$ be a point. By Lemma~\ref{lem:asymptoticproductprops} the system of subsets $\sheaff\ast\sheafg$ is a coarse ultrafilter on $X\ast Y$. Denote by $p_1':X\times Y\to X,p_2':X\times Y\to Y$ the projection to the first, second factor, respectively and by $i:X\ast Y\to X\times Y$ the inclusion. Then $p_1=p_1'\circ i,p_2=p_2'\circ i$. Since $i_*(\sheaff\ast\sheafg)=\{A\times B:A\in\sheaff,B\in\sheafg\}$ we obtain the relations $p_{1*}'i_*(\sheaff\ast\sheafg)\lambda \sheaff,p_{2*}'i_*(\sheaff\ast\sheafg)\lambda \sheafg$. Thus we have proved $\langle \corona{p_1},\corona{p_2}\rangle (\sheaff\ast\sheafg)=(\sheaff,\sheafg)$.

Now we prove $(\corona{p_1}(\sheaff))\ast(\corona{p_2}(\sheaff))=\sheaff$ for every point $\sheaff\in\corona{X\ast Y}$: Let $A\in\sheaff$ be an element. Then $(p_1(A)\times p_2(A))\cap (X\ast Y)\in (p_{1*}\sheaff)\ast(p_{2*}\sheaff)$. Since $A\s (p_1(A)\times p_2(A))\cap (X\ast Y)$ we obtain $(p_{1*}\sheaff)\ast(p_{2*}\sheaff)\s \sheaff$. Thus $(p_{1*}\sheaff)\ast(p_{2*}\sheaff)\lambda \sheaff$. This way we have shown $\langle \corona{p_1},\corona{p_2}\rangle$ is bijective.

Since $\corona{X\ast Y}$ is compact and $\corona X\ast\corona Y$ is Hausdorff we obtain that $\langle \corona{p_1},\corona{p_2}\rangle$ is a homeomorphism.
\end{proof}

\begin{lem}
Let $X,Y$ be metric spaces. There is a homeomorphism 
\f{
\corona X \times \corona Y &\to \corona{X\times Y}\\
(\sheaff,\sheafg)&\mapsto \sheaff\times \sheafg
}
where $\sheaff\times \sheafg:=\{A\times B:A\in \sheaff,B\in\sheafg\}$. 
\end{lem}
\begin{proof}
We already showed in the proof of Lemma~\ref{lem:asymptoticproductprops} that $\sheaff\times\sheafg$ is a coarse ultrafilter on $X\times Y$. It remains to show that the map is bijective and continuous.

Let $\sheaff_1,\sheaff_2\in\corona X,\sheafg_1,\sheafg_2 \in \corona Y$ be coarse ultrafilters. Suppose $(\sheaff_1\times \sheafg_1)\lambda (\sheaff_2\times \sheafg_2)$. Let $A\in \sheaff_1,B\in \sheaff_2$ be elements. Then $A\times Y \in \sheaff_1\times\sheafg_1,B\times Y\in\sheaff_2\times \sheafg_2$. Thus $A\times Y\close B\times Y$. This implies $A\close B$ in $X$, thus $\sheaff_1\lambda \sheaff_2$.

Let $\sheaff\in \corona {X\times Y}$ be a coarse ultrafilter. Define
\[
\sheaff_i:=\{p_i(A):A\in\sheaff\}
\]
for $i=1,2$. Here $p_i$ denotes the projection to the $i$th factor. Then $\sheaff_1$ is a coarse ultrafilter on $X$:
\begin{enumerate}
\item If $A,B\in\sheaff_1$ then $A\times Y,B\times Y\in \sheaff$. This implies $A\close B$.
\item If $A,B\s X$ with $A\cup B\in\sheaff_1$ then $(A\cup B)\times Y\in \sheaff$. Thus $A\times Y\in\sheaff$ or $B\times Y\in \sheaff$. Then $A\in \sheaff_1$ or $B\in \sheaff_1$
\item Since $X\times Y\in\sheaff$ the set $X\in\sheaff_1$ is contained.
\end{enumerate}
Since $A\s p_1(A)\times p_2(A)$ we have an inclusion $\sheaff_1\times\sheaff_2\s \sheaff$. Thus $(\sheaff_1\times\sheaff_2)\lambda \sheaff$.

Fix a coarse ultrafilter $\sheafg\in\corona Y$. We show the map
\f{
\corona X &\to \corona {X\times Y}\\
\sheaff&\mapsto \sheaff\times \sheafg 
}
is continuous: Let $\pi_1,\pi_2\s \corona X$ be subsets with $(\pi_1\times \sheafg)\notclose (\pi_2\times \sheafg)$. Then there exist subsets $A,B\s X\times Y$ with $\pi_1\times\sheafg\s \closedop A,\pi_2\times \sheafg\s\closedop B$ and $A\notclose B$. Since the left side is a product we can assume $A=A_1\times A_2,B=B_1\times B_2$ also. Then $\pi_1\s \closedop {A_1},\pi_2\s \closedop {B_1}$ with $A_1\notclose B_1$.
\end{proof}

If $X$ is a metric space we associate to $X$ a Grothendieck topology determined by coarse covers. Sheaf cohomology on coarse covers is coined coarse cohomology with twisted coefficients in \cite{Hartmann2017a}. Now coarse covers on $X$ determine the finite open covers on $\corona X$. Thus sheaf cohomology on $\corona X$ equals twisted cohomology on $X$ as a coarse space. We compose a Künneth formula for coarse cohomology with twisted coefficients.

\begin{thm}\name{Künneth formula}
Let $X,Y$ be metric spaces, $\sheaff$ a sheaf on $X$ and $\sheafg$ a sheaf on $Y$. Define a presheaf $\sheaff'$ on $X\times Y$ by
\[
U\mapsto \sheaff(p_1(U)).
\]
Then $\sheaff'$ is a sheaf on $X\times Y$. Similarly we can define a sheaf $\sheafg'$ on $X\times Y$. There is a homomorphism
\[
\bigoplus_{p+q=n}\cohomology p X \sheaff \otimes \cohomology q Y \sheafg \to \cohomology n {X\times Y} {\sheaff'\otimes \sheafg'}  
\]
Here $\sheaff'\otimes\sheafg'$ denotes the sheaf associated to the presheaf $U\mapsto \sheaff'(U)\otimes \sheafg'(U)$ for $U\s X\times Y$. If there is a $\sheaff$-acyclic coarse cover $\ucover$ of $X$ and a $\sheafg$-acyclic coarse cover $\vcover$ of $Y$ such that $\check C^q(\vcover,\sheafg)$ is torsion free for every $q$ and $\cohomology p \ucover\sheaff$ is torsion free for every $p$ then the homomorphism is an isomorphism.
\end{thm}
\begin{proof}
There is a \v Cech cohomology version of the Eilenberg-Zilber theorem. If $\ucover,\vcover$ are coarse covers of $X,Y$, respectively then 
\[
\ucover\times\vcover:=\{U_i\times V_i:U_i\in\ucover,V_j\in \vcover\}
\]
is a coarse cover of $X\times Y$. Then there is a homomorphism
\[
\bigoplus_{p+q=n}\check C^p(\ucover,\sheaff)\otimes \check C^q(\vcover,\sheafg)
\to \check C^n(\ucover\times \vcover,\sheaff'\otimes\sheafg')
\]
for every $n\ge 0$ which maps $(s_{i_0\cdots i_p})\in\prod\sheaff(U_{i_0}\cap\cdots\cap U_{i_p}),(t_{j_0\cdots j_q})\in\prod\sheafg(V_{j_0}\cap\cdots\cap V_{j_q})$ to $(s_{i_0\cdots i_p}\otimes t_{j_0\cdots j_q})\in \prod(\sheaff'\otimes\sheafg')((U_{i_0}\cap\cdots\cap U_{i_p})\times (V_{j_0}\cap\cdots\cap  V_{j_q})$. This induces an isomorphism of cochain complexes. We can now apply \cite[Section~2.8, Chapter~1]{Shafarevich1996} which gives the desired result in case of acyclic coarse covers. In the other case taking the direct limit over coarse covers gives the desired homomorphism.
\end{proof}

\section{Space of Rays}

\begin{defn}\name{space of rays}
 Let $Y$ be a compact topological space. As a set the \emph{space of rays} $\digamma(Y)$ of $Y$ is $Y\times \Z_+$. A subset $E\s Y^2$ is an entourage if for every countable subset $((x_k,i_k),(y_k,j_k))_k\s E$ the following properties hold:
 \begin{enumerate}
 \item The set $(i_k,j_k)_k$ is an entourage in $\Z_+$.
 \item If $(i_k)_k\to \infty$ then $(x_k)_k$ and $(y_k)_k$ have the same limit points.
 \end{enumerate}
 This makes $\digamma(Y)$ a coarse space.
\end{defn}

\begin{thm}
 If $f:X\to Y$ is a continuous map between compact topological spaces
 \begin{itemize}
 \item then it induces a coarse map by
 \f{
  \digamma(f):\digamma(X)&\to \digamma(Y)\\
  (x,i)&\mapsto (f(x),i)
}
\item If $f$ is a homeomorphism then $\digamma(f)$ is a coarse equivalence.
\end{itemize}
\end{thm}
\begin{proof}
 \begin{itemize}
\item We show $\digamma(f)$ is coarsely uniform and coarsely proper. First we show $\digamma(f)$ is coarsely uniform: Suppose $((x_i,n_i),(y_i,m_i))_i$ is a countable entourage in $\digamma(X)$ such that $(n_i)_i$ is a strictly monotone sequence in $\Z_+$ and $(x_i)_i$ converges to $x$. Then $(n_i,m_i)_i$ is an entourage in $\Z_+$ and $(y_i)_i$ converges to $x$. Since $f$ is a continuous map $f(x_i)_i$ and $f(y_i)_i$ both converge to $f(x)$. Thus we can conclude that
\[
((f(x_i),n_i),(f(y_i),m_i))_i
\]
is an entourage in $\digamma(Y)$.

Now we show $\digamma(f)$ is coarsely proper: If $B\s \digamma(Y)$ is bounded we can write $B=\bigcup_i B_i\times i$ with $B_i\s Y,i\in\Z_+$ where the number of $i$ that appear is finite. Then
\[
 \iip f B=\bigcup_i \iip f {B_i}\times i
\]
is bounded.
\item if $f$ is a homeomorphism then there is a topological inverse $g:Y\to X$ of $f$. Now $f\circ g=id_Y$ and $g\circ f=id_X$. Then
\f{
\digamma(f)\circ \digamma(g)
&=\digamma(f\circ g)\\
&=\digamma(id_Y)\\
&=id_{\digamma(Y)}
}
and
\f{
\digamma(g)\circ \digamma(f)
&=\digamma(g\circ f)\\
&=\digamma(id_X)\\
&=id_{\digamma(X)}
}
 \end{itemize}
\end{proof}

\begin{cor}
 Denote by $\ktopology$ the category of compact topological spaces and continuous maps and by $\coarse$ denote the category of coarse spaces and coarse maps modulo close. Then $\digamma$ is a functor
 \[
  \digamma:\ktopology\to \coarse
 \]
\end{cor}

\begin{prop}
\label{prop:unit}
Denote by $\sheaff_0$ a coarse ultrafilter on $\Z_+$, the choice is not important. For every $y\in Y$ denote by $i_y$ the inclusion $y\times\Z_+\to \digamma(Y)$. The map
\f{
 \eta_Y:Y&\to \nu'\circ \digamma(Y)\\
 y&\mapsto \corona{i_y}(\sheaff_0)
 }
for every metric space $Y$ defines a natural transformation $\eta:\mathbb{1}_{\ktopology} \to \nu'\circ \digamma$.
\end{prop}
\begin{proof}
 If $f:Y\to Z$ is a continuous map between compact spaces we show the diagram
\[
 \xymatrix{
 Y\ar[r]^f\ar[d]_{\eta_Y}
 & Z\ar[d]^{\eta_Z}\\
 \nu'\circ \digamma(Y)\ar[r]_{\nu'\circ \digamma(f)}
 & \nu'\circ \digamma(Z)
 }
\]
commutes.
down and then right: a point $y\in Y$ is mapped by $\eta_Y$ to $\corona{i_y}(\sheaff_0)$. Then
\f{
\nu'\circ \digamma(f)(\corona{i_y}(\sheaff_0))
&=\digamma(f)_*\circ i_{y*}(\sheaff_0)\\
&=(\digamma(f)\circ i_y)_*(\sheaff_0)\\
&=i_{f(y)*}(\sheaff_0)
}
right and then down: a point $y\in Y$ is mapped by $f$ to $f(y)$. Then
\[
 \eta_Z(f(y))=\corona{i_{f(y)}}(\sheaff_0)
\]
 The map $\eta_Y$ is continuous for every compact space $Y$: Let $(y_i)_i$ be a net in $Y$ that converges to $y$. Then $(\corona{i_{y_i}}(\sheaff_0))_i$ converges in $\eta_Y(Y)$ to $\corona{i_y}(\sheaff_0)$: Let $A\s \nu'\circ \digamma (Y)$ be a set such that $\corona{i_y}(\sheaff_0)\in\closedop A^c$. Thus there is some $B\in\sheaff_0$ such that $y\times B\notclose A$. Now for almost all $i$ the relation $(y_i\times B)\notclose A$ holds, thus $\corona {i_{y_i}}(\sheaff_0)\in\closedop A^c$ for almost all $i$.
\end{proof}

\begin{lem}
\label{lem:curay}
Let $X$ be a coarsely geodesic coarsely proper metric space. If $\sheaff$ is a coarse ultrafilter on $X$ there is a coarsely injective coarse map $\rho:\Z_+\to X$ such that $\sheaff\in \closedop {\rho(\Z_+)}$.
\end{lem}
\begin{proof}
Fix an ultrafilter $\sigma$ on $X$ that induces the coarse ultrafilter $\sheaff$. Suppose $X$ is $R$-discrete and $c$-coarsely geodesic for $R,c>0$. We will determine a sequence $(r_i)_i$ of points in $X$ and a sequence $(V_i)_i$ of subsets of $X$. 

Fix a point $x_0\in X$ and define $r_0:=x_0$ and $V_0:=X$. Then define for every $y\in X$ the number $d_0(y)$ to be the minimal length of a $c$-path joining $x_0$ to $y$. We define a relation on points of $X$: $y\le z$ if $d_0(y)\le d_0(z)$ and $y$ lies on a $c$-path of minimal length joining $x_0$ to $z$.

For every $i\in \N$ do: Denote by $C_i:=\{y\in X:d_0(y)=i\}$ and define $W_y:=\{z:y\le z\}\cap V_{i-1}$ for every $y\in C_i\cap V_{i-1}$. Now $V_{i-1}\in \sigma$ and the $W_y$ cover $V_{i-1}$ except for a bounded set. Then there is one $y$ such that $W_y\in \sigma$. Define $V_i:=W_y$ and $r_i:=y$.

Define a map 
\f{
\rho:\Z_+&\to X\\
i&\mapsto r_i.
}
Then $\rho$ is a coarsely injective coarse map with $(\rho(\Z_+))\in \sigma$.
\end{proof}

\section{An alternative Description of the Gromov Boundary}

\begin{thm}
Let $X$ be a proper geodesic hyperbolic metric space. The relation $\sheaff\sim\sheafg$ if $\sheaff,\sheafg\in \closedop {\rho(\Z_+)}$ for some coarsely injective coarse map $\rho:\Z_+\to X$ is an equivalence relation on coarse ultrafilters and the mapping
\f{
q_X:\corona X&\to\partial (X)\\
\sheaff&\mapsto [\rho]\;\;\; \sheaff\in\closedop{\rho(\Z_+)}
}
to the Gromov boundary $\partial (X)$ of $X$ is continuous and respects $\sim$. The induced map on the quotient associated to $\sim$ is a homeomorphism.

If $A\s X$ is a subset then
\[
\partial_X A:=\{[\rho]:\rho(\Z_+)\close A\}
\]
is closed in $\partial (X)$. The $((\partial_X A)^c)_{A\s X}$ constitute a basis for the topology on $\partial (X)$. 
\end{thm}
\begin{proof}
Note the first part is already \cite[Lemma~6.23]{Roe2003} which shows the Gromov boundary appears as a quotient of the Higson corona by using the property that every Gromov function is a Higson function. The second part is already \cite[Theorem~9.10] {Grzegrzolka2018a} which defines a coarse proximity structure on $X$ that induces the Gromov compactification. 

Every point $p$ in the Gromov boundary $\partial (X)$ is represented by a coarsely injective coarse map $\rho:\Z_+\to X$: A point in $\partial (X)$ is represented by a geodesic ray $r:\R_+\to X$ as defined in \cite[page~427]{Bridson1999}. By \cite[Lemma~3.1]{Bridson1999} the point $p$ can be represented by a large-scale embedding $\Z_+\to X$. Since $\Z_+,X$ are large-scale geodesic this is the same as a coarsely injective coarse map.

If $\rho,\sigma:\Z_+\to X$ are two coarsely injective coarse maps then either $\rho(\Z_+),\sigma(\Z_+)$ are finite Hausdorff distance apart or $\rho(\Z_+)\notclose\sigma(\Z_+)$: Suppose $\rho(\Z_+)\close\sigma(\Z_+)$. Then there are subsequences $(j_i)_i,(k_i)_i\s \Z_+$ and a constant $R\ge 0$ such that $d(\rho(j_i),\sigma(k_i))\le R$ for every $i$. By \cite[Theorem~6.17]{Roe2003} there exists $S>0$ such that the geodesic joining $\rho(j_i)$ to $\rho(j_{i+1})$ has Hausdorff distance at most $S$ from $\rho(j_i),\rho(j_i+1),\ldots,\rho(j_{i+1})$ and from $\sigma(k_i),\sigma(k_i+1),\ldots,\sigma(k_{i+1})$ for every $i$. Thus we obtain $d(\rho(\Z_+),\sigma(\Z_+))\le 2S$.

By Lemma~\ref{lem:curay} for every coarse ultrafilter $\sheaff$ there exists a coarsely injective coarse map $\rho:\Z_+\to X$ such that $\sheaff\in\closedop {\rho(\Z_+)}$. This implies $\sim $ is an equivalence relation on coarse ultrafilters. Since the equivalence classes are closed the quotient is T1.

We recall \cite[Definition~6.21]{Roe2003}: If $\varphi:X\to \R$ is a continuous function then it is called \emph{Gromov} if for every $\varepsilon>0$ there exists $K>0$ such that $(x|y)_{x_0}>K$ implies $|f(x)-f(y)|<\varepsilon$. We denote by $C_g(X)$ the algebra of Higson functions on $X$.

Now we provide the mapping $q_X$. Note that by \cite[Lemma~6.23]{Roe2003} every Gromov function is a Higson function. Thus there is a mapping
\f{
\Phi_X:\R^{C_h(X)}&\to\R^{C_g(X)}\\
(x_\varphi)_{\varphi\in C_h(X)}&\mapsto(x_\varphi)_{\varphi\in C_g(X)}.
}
Now $\Phi_X(\overline{e_{C_h(X)}(X)}\ohne e_{C_h(X)}(X))\s\overline{e_{C_g(X)}(X)}\ohne e_{C_g(X)}(X)$. In fact this map is surjective. This map associates a net $(x_i)_i$ that is section of a coarse ultrafilter to a net $(x_i)_i$ such that $\lim_i\varphi(x_i)\in\R$ exists for every Gromov function $\varphi$. By \cite[Lemma~6.24]{Roe2003} every such net arises as $\rho(i)_i$ for some coarsely injective coarse map $\rho:\Z_+\to X$. Thus $\rho(i)_i$ is a section of some ultrafilter inducing $\sheaff$ which translates to $\sheaff\in\closedop{\rho(\Z_+)}$. Note the map $q_X$ maps $\sheaff$ to $[\rho]\in\partial (X)$. 

Now $q_X$ respects $\sim$ and by the above it induces a continuous bijection $\corona X/\sim\to \partial (X)$.

We show the second part of the theorem: Denote by $q:\corona X\to \corona X/\sim $ the quotient map associated to $\sim$. Then
\f{
\iip q {\closedop A}
&=\{[\sheaff]:\sheaff\in\closedop A\}\\
&=\{\closedop{\rho(\Z_+)}:\sheaff\in\closedop{\rho(\Z_+)},\sheaff\in \closedop A\}\\
&=\{\closedop{\rho(\Z_+)}:\rho(\Z_+)\close A\}.
}
Then $\{[\rho]:\rho(\Z_+)\close A\}$ is closed in $\partial (X)$. The $\partial_X A=\iip {q_X}{\closedop A}$ generate the closed sets of $\partial (X)$.

We define a topology on $gX:=X\cup \partial (X)$ by declaring
\[
(\bar A\cup \partial_X A)^c
\]
as a base. Then $gX$ is compact: Let $(x_i)_i$ be a net in $gX$. If $(x_i)_i\cap X$ contains a bounded and infinite subsequence then there is a limit point $x\in X$ to which a subsequence converges. If this is not the case and $(x_i)_i\cap X$ is infinite then by \cite[Proposition~22]{Hartmann2019b} there exists a coarsely injective coarse map $\rho:\Z_+\to X$ with $\rho(\Z_+)\close ((x_i)_i\cap X)$. Then a subsequence converges to $[\rho]$. If $(x_i)_i\cap X$ is finite then a subnet of $(x_i)_i$ converges to a point in $\partial (X)$ since $\partial (X)$ is compact.

Now $X,\partial (X)$ appear as subspaces of $gX$. Since $\bar X^g=gX$ the space $gX$ is a compactification of $X$.
\end{proof}

\begin{cor}
If $f:X\to Y$ is a coarse map between hyperbolic proper geodesic metric spaces and if for every coarsely injective coarse map $\rho:\Z_+\to X$ the map 
\[
f\circ\rho:\Z_+\to Y
\]
is coarsely injective then $f$ induces a continuous map $\partial (f):\partial (X)\to \partial (Y)$.
\end{cor}
\begin{proof}
Compare this result with \cite[Theorem~2.8]{Dydak2016} where a visual large-scale uniform map induces a continuos map between Gromov boundaries.

Note that $\corona f$ maps equivalence classes of $\sim$ in $\corona X$ to equivalence classes of $\sim$ in $\corona Y$. Thus if $\sheaff\sim\sheafg$ in $\corona X$ then $q_Y\circ \corona f(\sheaff)=q_Y\circ \corona f(\sheafg)$. This implies there is a unique continuous map $\tilde f:\partial (X)\to\partial (Y)$ such that the following diagram commutes:
 \[
 \xymatrix{
 \corona X\ar[r]^{\corona f}\ar[d]_{q_X} 
 & \corona Y\ar[d]^{q_Y}\\
 \partial (X)\ar[r]_{\tilde f}
 &\partial (Y)
 }
 \]
 Now the map 
 \f{
 \partial (f):\partial X&\to \partial Y\\
 [\rho]&\mapsto [f\circ \rho]
 }
 also makes this diagram commute, thus $\partial (f)=\tilde f$ is continuous by uniqueness.
\end{proof}

\bibliographystyle{utcaps}
\bibliography{mybib}

\providecommand{\href}[2]{#2}\begingroup\raggedright\begin{thebibliography}{10}

\bibitem{Protasov2003}
I.~V. Protasov, ``Normal ball structures,'' {\em Mat. Stud.} {\bfseries 20}
  no.~1, (2003) 3--16.

\bibitem{Protasov2005}
I.~V. Protasov, ``Coronas of balleans,''
  \href{http://dx.doi.org/10.1016/j.topol.2004.09.005}{{\em Topology Appl.}
  {\bfseries 149} no.~1-3, (2005) 149--160}.
  \url{https://doi.org/10.1016/j.topol.2004.09.005}.

\bibitem{Protasov2011}
I.~V. Protasov, ``Coronas of ultrametric spaces,'' {\em Comment. Math. Univ.
  Carolin.} {\bfseries 52} no.~2, (2011) 303--307.

\bibitem{Banakh2013}
T.~Banakh, O.~Chervak, and L.~Zdomskyy, ``On character of points in the
  {H}igson corona of a metric space,'' {\em Comment. Math. Univ. Carolin.}
  {\bfseries 54} no.~2, (2013) 159--178.

\bibitem{Protasov2015}
I.~V. Protasov and S.~V. Slobodianiuk, ``Ultrafilters on balleans,''
  \href{http://dx.doi.org/10.1007/s11253-016-1200-y}{{\em Ukra\"{i}n. Mat. Zh.}
  {\bfseries 67} no.~12, (2015) 1698--1706}.
  \url{https://doi.org/10.1007/s11253-016-1200-y}. Reprinted in Ukrainian Math.
  J. {{\bf{6}}7} (2016), no. 12, 1922--1931.

\bibitem{Grzegrzolka2018a}
P.~{Grzegrzolka} and J.~{Siegert}, ``{Boundaries of coarse proximity spaces and
  boundaries of compactifications},'' {\em arXiv e-prints} (Dec, 2018)
  arXiv:1812.09802, \href{http://arxiv.org/abs/1812.09802}{{\ttfamily
  arXiv:1812.09802 [math.GN]}}.

\bibitem{Hartmann2019a}
E.~{Hartmann}, ``{Twisted Coefficients on coarse Spaces and their Corona},''
  {\em arXiv e-prints} (Mar, 2019) arXiv:1904.00380,
  \href{http://arxiv.org/abs/1904.00380}{{\ttfamily arXiv:1904.00380
  [math.MG]}}.

\bibitem{Baladze2018}
V.~{Baladze} and F.~{Dumbadze}, ``{On (Co)homological Properties of Stone-Cech
  Compactifications of Completely Regular Spaces},'' {\em arXiv e-prints} (Jun,
  2018) arXiv:1806.01566, \href{http://arxiv.org/abs/1806.01566}{{\ttfamily
  arXiv:1806.01566 [math.AT]}}.

\bibitem{Roe2003}
J.~Roe, \href{http://dx.doi.org/10.1090/ulect/031}{{\em Lectures on coarse
  geometry}}, vol.~31 of {\em University Lecture Series}.
\newblock American Mathematical Society, Providence, RI, 2003.
\newblock \url{http://dx.doi.org/10.1090/ulect/031}.

\bibitem{Hartmann2019b}
E.~{Hartmann}, ``{A pullback diagram in the coarse category},'' {\em arXiv
  e-prints} (Jul, 2019) arXiv:1907.02961,
  \href{http://arxiv.org/abs/1907.02961}{{\ttfamily arXiv:1907.02961
  [math.MG]}}.

\bibitem{Foertsch2003}
T.~Foertsch and V.~Schroeder, ``Products of hyperbolic metric spaces,''
  \href{http://dx.doi.org/10.1023/B:GEOM.0000006539.14783.aa}{{\em Geom.
  Dedicata} {\bfseries 102} (2003) 197--212}.
  \url{https://doi.org/10.1023/B:GEOM.0000006539.14783.aa}.

\bibitem{Benakli2002}
I.~Kapovich and N.~Benakli,
  \href{http://dx.doi.org/10.1090/conm/296/05068}{``Boundaries of hyperbolic
  groups,''} in {\em Combinatorial and geometric group theory ({N}ew {Y}ork,
  2000/{H}oboken, {NJ}, 2001)}, vol.~296 of {\em Contemp. Math.}, pp.~39--93.
\newblock Amer. Math. Soc., Providence, RI, 2002.
\newblock \url{http://dx.doi.org/10.1090/conm/296/05068}.

\bibitem{Kalantari2016}
S.~Kalantari and B.~Honari, ``Asymptotic resemblance,'' {\em Rocky Mountain J.
  Math.} {\bfseries 46} no.~4, (2016) 1231--1262.
  \url{https://doi.org/10.1216/RMJ-2016-46-4-1231}.

\bibitem{Riehl2017}
E.~Riehl, {\em Category theory in context}.
\newblock Courier Dover Publications, 2017.

\bibitem{Cornulier2016}
Y.~Cornulier and P.~de~la Harpe, \href{http://dx.doi.org/10.4171/166}{{\em
  Metric geometry of locally compact groups}}, vol.~25 of {\em EMS Tracts in
  Mathematics}.
\newblock European Mathematical Society (EMS), Z\"urich, 2016.
\newblock \url{http://dx.doi.org/10.4171/166}.
\newblock Winner of the 2016 EMS Monograph Award.

\bibitem{Ball1983}
B.~J. Ball and S.~Yokura, ``Compactifications determined by subsets of
  {$C\sp{\ast} (X)$}. {II},''
  \href{http://dx.doi.org/10.1016/0166-8641(83)90041-X}{{\em Topology Appl.}
  {\bfseries 15} no.~1, (1983) 1--6}.
  \url{https://doi.org/10.1016/0166-8641(83)90041-X}.

\bibitem{Dranishnikov1998}
A.~N. Dranishnikov, J.~Keesling, and V.~V. Uspenskij, ``On the {H}igson corona
  of uniformly contractible spaces,''
  \href{http://dx.doi.org/10.1016/S0040-9383(97)00048-7}{{\em Topology}
  {\bfseries 37} no.~4, (1998) 791--803}.
  \url{http://dx.doi.org/10.1016/S0040-9383(97)00048-7}.

\bibitem{Hartmann2017a}
E.~{Hartmann}, ``{Coarse Cohomology with twisted Coefficients},'' {\em ArXiv
  e-prints} (Sept., 2017) , \href{http://arxiv.org/abs/1710.06725}{{\ttfamily
  arXiv:1710.06725 [math.AG]}}.

\bibitem{Shafarevich1996}
I.~R. Shafarevich, {\em Algebraic Geometry 2}.
\newblock Springer-Verlag, Berlin Heidelberg, 1996.

\bibitem{Bridson1999}
M.~R. Bridson and A.~Haefliger,
  \href{http://dx.doi.org/10.1007/978-3-662-12494-9}{{\em Metric spaces of
  non-positive curvature}}, vol.~319 of {\em Grundlehren der Mathematischen
  Wissenschaften [Fundamental Principles of Mathematical Sciences]}.
\newblock Springer-Verlag, Berlin, 1999.
\newblock \url{http://dx.doi.org/10.1007/978-3-662-12494-9}.

\bibitem{Dydak2016}
J.~Dydak and v.~Virk, ``Inducing maps between {G}romov boundaries,''
  \href{http://dx.doi.org/10.1007/s00009-015-0650-z}{{\em Mediterr. J. Math.}
  {\bfseries 13} no.~5, (2016) 2733--2752}.
  \url{http://dx.doi.org/10.1007/s00009-015-0650-z}.

\end{thebibliography}\endgroup

\address

\end{document}